\documentclass[a4paper,10pt]{amsart}
\usepackage{a4wide}
\usepackage{url}
\usepackage{amssymb}
\usepackage{booktabs}          %needed for a better control of tables and arrays
\usepackage{float}             %needed for the option H in tables (here, absolutely here!)
\usepackage{hyperref}
\IfFileExists{\jobname.TEX}{}{\usepackage{srcltx}\usepackage[utf8]{inputenc}}

\numberwithin{equation}{section}

\newcommand{\field}[1]{\mathbb{{#1}}}

\newcommand{\R}{\field{R}}
\newcommand{\C}{\field{C}}
\newcommand{\N}{\field{{N}}}

\newcommand{\intpart}[1]{\left\lfloor#1\right\rfloor}

\newcommand{\intpartup}[1]{\left\lceil#1\right\rceil}
\renewcommand{\d}{\,\mathrm{d}}
\newcommand{\Imm}{\mathrm{Im}}
\newcommand{\Ree}{\mathrm{Re}}

\newcommand{\Rp}{U}

\DeclareMathOperator{\sinc}{sinc}

\newtheorem{theorem}{Theorem}[section]
\newtheorem*{theorem*}{Theorem}
\newtheorem{lemma}[theorem]{Lemma}
\newtheorem*{lemma*}{Lemma}
\newtheorem{proposition}[theorem]{Proposition}
\newtheorem{corollary}[theorem]{Corollary}
\newtheorem*{corollary*}{Corollary}

\theoremstyle{remark}

\newtheorem*{remark*}{Remark}
\newtheorem*{acknowledgements}{Acknowledgements}

\allowdisplaybreaks[4]

\begin{document}
\title{Primes in explicit short intervals on RH}

\author[A.~Dudek]{Adrian W. Dudek}
\address[A.~Dudek]{Mathematical Sciences Institute\\
         The Australian National University}
\email{adrian.dudek@anu.edu.au}

\author[L.~Greni\'{e}]{Lo\"{\i}c Greni\'{e}}
\address[L.~Greni\'{e}]{Dipartimento di Ingegneria gestionale, dell'informazione e della produzione\\
         Universit\`{a} di Bergamo\\
         viale Marconi 5\\
         24044 Dalmine (BG)
         Italy}
\email{loic.grenie@gmail.com}

\author[G.~Molteni]{Giuseppe Molteni}
\address[G.~Molteni]{Dipartimento di Matematica\\
         Universit\`{a} di Milano\\
         via Saldini 50\\
         20133 Milano\\
         Italy}
\email{giuseppe.molteni1@unimi.it}

%11N05 =  Distribution of primes

%\date{\today. File name: {\tt \jobname.tex}}

\begin{abstract}
In this paper, on the assumption of the Riemann hypothesis, we give explicit upper bounds on the
difference between consecutive prime numbers.
\end{abstract}

\maketitle

\begin{center}
To appear in Int. J. Number Theory. 2015.
%DOI: 10.1142/S1793042116500858
\end{center}

\section{General setting and results}
\label{sec:A1}

The computation of the maximal prime gaps given by Oliveira e Silva, Herzog and
Pardi~\cite[Sec.~2.2]{Oliveira-SilvaHerzogPardi} verifies that $p_{k+1} - p_k < \log^2 p_k$ for all
primes $11\leq p_k \leq 4 \cdot 10^{18}$. This proves that $\forall x\in[5,4\cdot 10^{18}]$, there is a
prime in $[x-0.5\log^2 x,x+0.5\log^2 x]$.
It is the purpose of this article to furnish new explicit upper bounds on the difference between
consecutive prime numbers with the assumption of the Riemann hypothesis. Specifically, we prove the
following theorem.

\begin{theorem}\label{th:A1}
Assume RH. Let $x\geq 2$ and $c := \frac{1}{2} + \frac{2}{\log x}$. Then there is a prime in
$(x-c\sqrt{x}\log x,x+c\sqrt{x}\log x)$ and at least $\sqrt{x}$ primes in $(x-(c+1)\sqrt{x}\log
x,x+(c+1)\sqrt{x}\log x)$.
\end{theorem}

The mentioned conclusion coming from the computations of Oliveira e Silva et al{.} is stronger than
the first part of our result for all $x \leq 4\cdot 10^{18}$.
This allows one to use $c=0.55$ for all $x\geq 5$ when only one prime is needed.\\
%G 0.5 + 2/log(4*10^18)
%G \\%59 = 0.5466931600162255414224421196
%
In a recent paper~\cite{Dudek}, the first author proved Theorem \ref{th:A1} with $c= \frac{2}{\pi} =
0.6366\ldots$ and an asymptotic result in the weaker form $c=0.5+\epsilon$ when $x\geq x(\epsilon)$,
without any information on the size of $x(\epsilon)$.

In Appendix~\ref{app:A1} we prove the same result with $c=0.6102$. This value improves on the one
in~\cite{Dudek} and is stronger than Theorem~\ref{th:A1} up to $2\cdot 10^{8}$. Despite its weakness, we
believe that its method of proof is worthy of interest. Indeed, the conclusion is reached proving that
the Riemann hypothesis implies a (very weak) cancellation in the exponential sum $S_\alpha(T):=
\sum_{|\gamma|\leq T} e^{i \alpha \gamma}$, where $\gamma$ runs on the set of imaginary parts of the
nontrivial zeros for the Riemann zeta function. Hypotheses ensuring the existence of stronger
cancellations produce stronger conclusions, and in fact one can show that for all $x$ there is a prime
$p$ such that $|p-x|=o(\sqrt{x\log x})$ assuming the Pair correlation
conjecture~\cite{Heath-BrownGoldston}, and even a prime $p$ with $|p-x|\ll_\epsilon x^\epsilon$ assuming
the stronger Gonek conjecture~\cite{Gonek}. This shows that the method in appendix has a good track
record and could represent the main path for this kind of results. The third author admits his surprise
when we did not manage to prove the best constant $c = 1/2 + \epsilon$ using this method.
\smallskip

We first consider the setting in which we seek to establish Theorem~\ref{th:A1}. Throughout, we define
the von Mangoldt function as
\[
\Lambda(n)
 := \left\{
    \begin{array}{ll}
       \log p  & : \hspace{2mm} n=p^m, \text{ $p$ is prime, $m \in \N$, $m\geq 1$}\\
           0   & : \hspace{2mm}        \text{otherwise},
    \end{array}
    \right.
\]
$\vartheta(x) := \sum_{p \leq x} \log p$, where the sum is restricted to primes, and $\psi(x) := \sum_{n
\leq x} \Lambda(n)$. It is often convenient to work with a smoothed version of $\psi$, and so we define
\[
\psi^{(1)}(x):= \int_0^x \psi(u)\d u
              = \sum_{n\leq x} \Lambda(n)(x-n)
\]
through partial summation. One can recall the integral representation
\[
\psi^{(1)}(x)
=  -\frac{1}{2\pi i}\int_{2-i\infty}^{2+i\infty} \frac{\zeta'}{\zeta}(s) \frac{x^{s+1}}{s(s+1)}\d s
\qquad
\forall x\geq 1
\]
which follows directly from an application of Perron's formula (see, for example, Ingham's classic
text~\cite[Ch~IV, Sec~4]{Ingham2}). We let $h \in \R$ such that $0<h<x$. Then
\[
\psi^{(1)}(x+h) - 2\psi^{(1)}(x) + \psi^{(1)}(x-h)
=
\sum_n \Lambda(n) K(x-n;h)
\]
where $K(u;h):=\max\{h-|u|,0\}$; one can verify this by expanding the left hand side of the above
identity. Note also that $K(u;h)$ is supported on $|u|\leq h$, positive in the open set, and has a unique
maximum at $u=0$ with $K(0;h)=h$.\\
From the integral representation one gets the explicit formula
\[
\psi^{(1)}(x)
=  \frac{x^2}{2}
 - \sum_{\rho} \frac{x^{\rho+1}}{\rho(\rho+1)}
 - x r
 +   r'
 + R^{(1)}(x)
\]
where $\rho$ runs on the set of nontrivial zeros of the Riemann zeta-function, $r$ and $r'$ are
constants, and $|R^{(1)}(x)|\leq 0.6/x$ (one can see~\cite[Lemma~3.3]{GrenieMolteni2}, though this is
classical). Noting that
\[
(x+h)^j - 2x^j + (x-h)^j
= \left\{
  \begin{array}{ll}
    0    & : \hspace{2mm} j=0,1\\
    2h^2 & : \hspace{2mm} j=2,
  \end{array}
  \right.
\]
we thus have that, assuming $h\leq x/\sqrt{3}$,
\[
\sum_n \Lambda(n) K(x-n;h)
= h^2
 - \sum_{\rho} \frac{(x+h)^{\rho+1}-2x^{\rho+1}+(x-h)^{\rho+1}}{\rho(\rho+1)}
 + \frac{3\theta}{x}
\]
for some $\theta=\theta(x,h)\in[-1,1]$, for then
\[
0.6((x+h)^{-1}+2x^{-1}+(x-h)^{-1}) \leq 3x^{-1}.
\]
We split the sum over the zeros as
\[
\sum_{\rho} \frac{(x+h)^{\rho+1}-2x^{\rho+1}+(x-h)^{\rho+1}}{\rho(\rho+1)}
=: \Sigma_1 + \Sigma_2,
\]
with $\Sigma_1$ and $\Sigma_2$ representing the sums on zeros with $|\Imm(\rho)|\leq T$ and
$|\Imm(\rho)|> T$, respectively. It is not a difficult task to bound $\Sigma_2$ (here we repeat the
argument in~\cite{Dudek}). In fact, assuming $RH$,
\[
|\Sigma_2| \leq  4(x+h)^{3/2} \sum_{|\Imm(\rho)|> T} \frac{1}{|\rho(\rho+1)|},
\]
and since $\sum_{|\Imm(\rho)|> T} \frac{1}{|\rho|^2}\leq \frac{\log T}{\pi T}$
(see~\cite[Lemma~1 (ii)]{SkewesII}), one has
\[
|\Sigma_2| \leq 4(x+h)^{3/2} \frac{\log T}{\pi T}.
\]
Thus
\[
\sum_n \Lambda(n) K(x-n;h)
\geq h^2
 - |\Sigma_1|
 - 4(x+h)^{3/2} \frac{\log T}{\pi T}
 - \frac{3}{x}.
\]
Now we remove the contribution from prime powers. Recalling that
\[
0.9986\sqrt{x} \leq \psi(x)-\vartheta(x)\leq (1 + 10^{-6})\sqrt{x} + 3\sqrt[3]{x}
\]
for every $x\geq 121$ (see~\cite[Th.~6]{RosserSchoenfeld2} and \cite[Cor.~2]{PlattTrudgian}), we get
\begin{align*}
\sum_n \Lambda(n) K(x-n;h)
&\leq h\sum_{|n-x|<h} \Lambda(n)                                                                  \\
&\leq h\Big(\sum_{|p-x|<h} \log p + (\psi(x+h)-\vartheta(x+h)) - (\psi(x-h)-\vartheta(x-h))\Big)  \\
&\leq h\Big(\sum_{|p-x|<h} \log p + (1+10^{-6})\sqrt{x+h} + 3\sqrt[3]{x+h} - 0.9986\sqrt{x-h}\Big)\\
&\leq h\Big(\sum_{|p-x|<h} \log p + 0.002\sqrt{x} + 3\sqrt[3]{x} + \frac{2h}{\sqrt{x}}\Big)
\end{align*}
where for the last inequality we have also used that $h\leq x/\sqrt{3}$.
%
% Let $y=h/x$, so
% (1+10^{-6})\sqrt{x+h} + 3\sqrt[3]{x+h} - 0.9986\sqrt{x-h}
% = \sqrt{x}((1+10^{-6})(1+y)^{1/2} - 0.9986(1-y)^{1/2}) + 3\sqrt[3]{x}(1+y)^{1/3}
% for 0<y<1/sqrt{3},
% \leq \sqrt{x}(0.002 + 1.05 y) + \sqrt[3]{x}(3+y)
% =    0.002\sqrt{x} + 3\sqrt[3]{x} + 1.05 \frac{h}{\sqrt{x}}  + \frac{h}{x^{2/3}}
% using x\geq 121
% \leq 0.002\sqrt{x} + 3\sqrt[3]{x} + 2 \frac{h}{\sqrt{x}}.
%
Thus, when $x\geq 121$ we have
\begin{equation}\label{eq:A1}
\sum_{|p-x|<h} \log p
\geq h
 - \frac{1}{h}|\Sigma_1|
 - 4(x+h)^{3/2} \frac{\log T}{\pi hT}
 - 0.002\sqrt{x} - 3\sqrt[3]{x} - \frac{2h}{\sqrt{x}}
 - \frac{3}{xh}.
\end{equation}
It is clear that the positivity of the right hand side guarantees the existence of at least one prime in
the interval $(x-h,x+h)$.

From before, we have that
\[
\Sigma_1 = \sum_{|\Imm(\rho)|\leq T} \frac{(x+h)^{\rho+1}-2x^{\rho+1}+(x-h)^{\rho+1}}{\rho(\rho+1)}.
\]
There are essentially two ways to bound $\Sigma_1$, both of them appearing already in~\cite{Dudek}.

The first one is based on the Taylor identity
\[
(1+\epsilon)^{\frac{3}{2}+i\gamma \epsilon} -2 + (1-\epsilon)^{\frac{3}{2}-i\gamma \epsilon}
= -4\sin^2(\gamma\epsilon) + O(\gamma\epsilon^2),
\]
while the second one is based on the identity
\[
\frac{(x+h)^{\rho+1} - 2x^{\rho+1} + (x-h)^{\rho+1}}{\rho(\rho+1)}
= \int_{x-h}^{x+h} K(x-u;h)u^{\rho-1}\d u.
\]
Thus, denoting $\gamma$ the imaginary part of a nontrivial zero, on the assumption of RH one gets
\begin{equation}\label{eq:A3}
|\Sigma_1| \leq  4\sum_{|\gamma|\leq T} \frac{\sin^2(\gamma\epsilon) + O(\gamma\epsilon^2)}{\gamma^2}
\end{equation}
from the first one, and
\begin{equation}\label{eq:A2}
|\Sigma_1| \leq  \int_{x-h}^{x+h} K(x-u;h)\Big|\sum_{|\gamma|\leq T} u^{i\gamma}\Big|\frac{\d u}{\sqrt{u}}
\end{equation}
from the second one. As a consequence, the first approach takes advantage of the cancellation due to the
sum of the three functions $(1+ \omega\epsilon)^{\frac{3}{2}\omega i\gamma \epsilon}$ with
$\omega\in\{0,\pm 1\}$ for the same zero, while the second approach takes advantage of the cancellation
coming from the sum of values of the same function computed at different zeros.\\
The first approach is discussed in Section~\ref{sec:A5}, while the second is discussed in
Appendix~\ref{app:A1}.

\section{First bound for $\Sigma_1$}
\label{sec:A5} Let $N(T)$ denote the number of nontrivial zeros of $\zeta(s)$ with imaginary part in
$[0,T]$, where multiplicity is included. We state the estimate of $N(T)$ done by Trudgian
in~\cite{TrudgianII}: let $W(T):= \frac{T}{2\pi}\log\big(\frac{T}{2\pi e}\big)$ denote what is
essentially the main term of $N(T)$ and let $\Rp(T) := N(T) - W(T)$, then the result says that
\[
|\Rp(T)| \leq 0.112\log T + 0.278\log\log T + 2.51 + \frac{0.2}{T} + \frac{7}{8} =: R(T)
\qquad T\geq e.
\]
Note that $\Rp(2\pi) = 1$ because the imaginary part of the first zero is $14.13\ldots$, and
that $\d W(T) = \log(\frac{T}{2\pi})\frac{\d T}{2\pi}$.

We introduce the notations
\[
T = \frac{\beta}{c}\frac{\sqrt{x}}{\log x},
\qquad
h = c\sqrt{x}\log x
\]
for suitable $\beta$ and $c$.

\begin{lemma}\label{lem:A1}
Let $0\leq h < x$. Then for every $\gamma\in\R$ there exists $\theta\in\C$ with $|\theta|\leq 1$ such
that
\[
\Big(1+\frac{h}{x}\Big)^{\frac{3}{2}+i\gamma} -2  + \Big(1-\frac{h}{x}\Big)^{\frac{3}{2}+i\gamma}
= -4\sin^2\Big(\frac{\gamma h}{2x}\Big) + \theta (2|\gamma|+1)\frac{h^2}{x^2}.
\]
\end{lemma}
\begin{proof}
The proof is straightforward and follows from the Taylor expansion of $\log(1+u)$ and some elementary
inequalities.
\end{proof}
Thus we get an explicit version of~\eqref{eq:A3}:
\begin{align*}
|\Sigma_1|
%&\leq 2x^{3/2}\sum_{0<\gamma\leq T} \frac{4\sin^2\big(\frac{\gamma h}{2x}\big)}{\gamma^2}
%    + 2x^{3/2}\frac{h^2}{x^2}\sum_{0<\gamma\leq T} \frac{2\gamma+1}{\gamma^2}\\
&\leq 8x^{3/2}\sum_{0<\gamma\leq T} \frac{\sin^2\big(\frac{\gamma h}{2x}\big)}{\gamma^2}
    + 2\frac{h^2}{\sqrt{x}}\sum_{0<\gamma\leq T} \frac{2\gamma+1}{\gamma^2}.
\end{align*}
From~\cite[Lemma 1 (i,iii)]{SkewesII}, we have that $\sum_{0<\gamma<T}\frac{1}{\gamma}\leq
\frac{\log^2 T}{4\pi}$ and $\sum_{\gamma>0} \frac{1}{\gamma^2}\leq \frac{1}{40}$ and thus
\begin{equation}\label{eq:A7}
|\Sigma_1|
\leq 8x^{3/2}\sum_{0<\gamma\leq T} \frac{\sin^2\big(\frac{\gamma h}{2x}\big)}{\gamma^2}
    + \Big(\frac{\log^2 T}{\pi}+\frac{1}{20}\Big)\frac{h^2}{\sqrt{x}}.
\end{equation}
Let $\gamma_1 = 14.13\ldots$ be the imaginary part of the first non-trivial zero of $\zeta(s)$. By
partial summation we get
\begin{align*}
\sum_{0<\gamma\leq T} \frac{\sin^2\big(\frac{\gamma h}{2x}\big)}{\gamma^2}
&= \int_{\gamma_1^-}^{T^+} \frac{\sin^2\big(\frac{\gamma h}{2x}\big)}{\gamma^2} \d N(\gamma)                 \\
&=\Big[\frac{\sin^2\big(\frac{\gamma h}{2x}\big)}{\gamma^2} N(\gamma)\Big]\Big|_{\gamma_1^-}^{T^+}
 -\int_{\gamma_1}^{T} \Big[\frac{\sin^2\big(\frac{\gamma h}{2x}\big)}{\gamma^2}\Big]' N(\gamma)\d\gamma    \\
&=\frac{\sin^2\big(\frac{hT}{2x}\big)}{T^2} N(T)
 -\int_{\gamma_1}^{T} \Big[\frac{\sin^2\big(\frac{\gamma h}{2x}\big)}{\gamma^2}\Big]' N(\gamma)\d\gamma.
\end{align*}
It then follows that
\begin{align*}
\sum_{0<\gamma\leq T} \frac{\sin^2\big(\frac{\gamma h}{2x}\big)}{\gamma^2}
={}&\frac{\sin^2\big(\frac{hT}{2x}\big)}{T^2} N(T)
 -\int_{\gamma_1}^{T} \Big[\frac{\sin^2\big(\frac{\gamma h}{2x}\big)}{\gamma^2}\Big]' \frac{\gamma}{2\pi}\log\Big(\frac{\gamma}{2\pi e}\Big)\d\gamma
 -\int_{\gamma_1}^{T} \Big[\frac{\sin^2\big(\frac{\gamma h}{2x}\big)}{\gamma^2}\Big]' \Rp(\gamma)\d\gamma  \\
%
%&=\frac{\sin^2\big(\frac{hT}{2x}\big)}{T^2} N(T)
% -\frac{\sin^2\big(\frac{\gamma h}{2x}\big)}{\gamma^2}\frac{\gamma}{2\pi}\log\Big(\frac{\gamma}{2\pi e}\Big)\Big|_{\gamma_1^-}^{T}
% +\int_{\gamma_1}^{T} \frac{\sin^2\big(\frac{\gamma h}{2x}\big)}{\gamma^2}\log\Big(\frac{\gamma}{2\pi}\Big)\frac{\d\gamma}{2\pi}
% -\int_{\gamma_1}^{T} \Big[\frac{\sin^2\big(\frac{\gamma h}{2x}\big)}
%{\gamma^2}\Big]' \Rp(\gamma)\d\gamma\\
%
={}&\frac{\sin^2\big(\frac{hT}{2x}\big)}{T^2} U(T)
 + \frac{\sin^2\big(\frac{\gamma_1 h}{2x}\big)}{2\pi\gamma_1}\log\Big(\frac{\gamma_1}{2\pi e}\Big)           \\
&+ \int_{\gamma_1}^{T} \frac{\sin^2\big(\frac{\gamma h}{2x}\big)}{\gamma^2}\log\Big(\frac{\gamma}{2\pi}\Big)\frac{\d\gamma}{2\pi}
 - \int_{\gamma_1}^{T} \Big[\frac{\sin^2\big(\frac{\gamma h}{2x}\big)}{\gamma^2}\Big]' \Rp(\gamma)\d\gamma.
\end{align*}
Recalling the upper bound $|\Rp(T)|\leq R(T)$ and noticing that $\gamma_1 < 2\pi e$, we get
\begin{align*}
\sum_{0<\gamma\leq T} \frac{\sin^2\big(\frac{\gamma h}{2x}\big)}{\gamma^2}
& \leq \frac{R(T)}{T^2}
 +\int_{\gamma_1}^{T} \frac{\sin^2\big(\frac{\gamma h}{2x}\big)}{\gamma^2}\log\Big(\frac{\gamma}{2\pi}\Big)\frac{\d\gamma}{2\pi}
 -\int_{\gamma_1}^{T} \Big[\frac{\sin^2\big(\frac{\gamma h}{2x}\big)}{\gamma^2}\Big]' \Rp(\gamma)\d\gamma    \\
%
%&\leq \frac{R(T)}{T^2}
% +\int_{\gamma_1}^{T} \frac{\sin^2\big(\frac{\gamma h}{2x}\big)}{\gamma^2} \log\Big(\frac{\gamma}{2\pi}\Big)\frac{\d\gamma}{2\pi}
% -\int_{\gamma_1}^{T} \Big[\frac{h}{2x}\frac{\sin\big(\frac{\gamma h}{x}\big)}{\gamma^2} - 2\frac{\sin^2\big(\frac{\gamma h}{2x}\big)}{\gamma^3}\Big] \Rp(\gamma)\d\gamma\\
%
&\leq \frac{R(T)}{T^2}
 +\int_{\gamma_1}^{T} \frac{\sin^2\big(\frac{\gamma h}{2x}\big)}{\gamma^2} \log\Big(\frac{\gamma}{2\pi}\Big)\frac{\d\gamma}{2\pi}
 +\int_{\gamma_1}^{T} \Big|\frac{h}{2x}\frac{\sin\big(\frac{\gamma h}{x}\big)}{\gamma^2} - 2\frac{\sin^2\big(\frac{\gamma h}{2x}\big)}{\gamma^3}\Big| R(\gamma)\d\gamma.
\intertext{Using the inequality $|\sin^2 v|\leq\frac{3}{4}|v|$, we simplify to get}
\sum_{0<\gamma\leq T} \frac{\sin^2\big(\frac{\gamma h}{2x}\big)}{\gamma^2}
&\leq \frac{R(T)}{T^2}
 +\int_{\gamma_1}^{T} \frac{\sin^2\big(\frac{\gamma h}{2x}\big)}{\gamma^2} \log\Big(\frac{\gamma}{2\pi}\Big)\frac{\d\gamma}{2\pi}
 +\frac{5h}{4x} \int_{\gamma_1}^{T} \frac{R(\gamma)}{\gamma^2}\d\gamma.
\intertext{Since $\int_{14}^{+\infty} \frac{R(\gamma)}{\gamma^2}\d\gamma\leq 0.297$, we get}
%
%G R(u) = 0.112*log(u)+0.278*log(log(u)) + 2.51 + 0.2/u + 7/8;
%G intnum(u=14,oo,R(u)/u^2)
%G \\%2 = 0.2964505932794089238835636091
%
\sum_{0<\gamma\leq T} \frac{\sin^2\big(\frac{\gamma h}{2x}\big)}{\gamma^2}
&\leq \frac{1}{2\pi}\int_{\gamma_1}^{T} \frac{\sin^2\big(\frac{\gamma h}{2x}\big)}{\gamma^2} \d\gamma \log\Big(\frac{T}{2\pi}\Big)
     + \frac{R(T)}{T^2}
     + \frac{2.97h}{8x}  \\
&\leq
 \frac{h}{4\pi x}\int_{0}^{\frac{hT}{2x}} \frac{\sin^2 t}{t^2} \d t \log\Big(\frac{T}{2\pi}\Big)
 +\frac{R(T)}{T^2}
 +\frac{2.97h}{8x}.
\end{align*}
We can bound the integral in the above equation with ease, for
\begin{align*}
\int_{0}^{y} \frac{\sin^2 t}{t^2} \d t
&%= \int_{0}^{\infty} \frac{\sin^2 t}{t^2} \d t - \int_{y}^{\infty} \frac{\sin^2 t}{t^2} \d t
 = \frac{\pi}{2} - \int_{y}^{\infty} \frac{\sin^2 t}{t^2} \d t
= \frac{\pi}{2} - \int_{y}^{\infty} \frac{1-\cos(2t)}{2t^2} \d t \\
& = \frac{\pi}{2} - \frac{1}{2y} + \int_{y}^{\infty} \frac{\cos(2t)}{2t^2} \d t
= \frac{\pi}{2} - \frac{1}{2y} - \frac{\sin(2y)}{4y^2} + \int_{y}^{\infty} \frac{\sin(2t)}{2t^3} \d t \\
& = \frac{\pi}{2} - \frac{1}{2y} - \frac{\sin(2y)}{4y^2} + \frac{\theta}{4y^2}
= \frac{\pi}{2} - \frac{1}{2y} + \frac{\theta}{2y^2}
\end{align*}
for some $\theta\in [-1,1]$. As such, we can now use $R(T)\leq 1.5 \log T$ for every $T\geq \gamma_1$
%
%G ploth(u=14.13,100,R(u)/log(u))
%G \\%6 = [14.130000000000000782, 100.00000000000000000, 0.9396688552296756924, 1.4977581138139277606]
%
to get
\begin{align*}
\sum_{0<\gamma\leq T}\frac{\sin^2\big(\frac{\gamma h}{2x}\big)}{\gamma^2}
\leq  \frac{h}{8 x}\Big(1 - \frac{2}{\pi}\frac{x}{hT} + \frac{4}{\pi}\frac{x^2}{h^2T^2}\Big)\log\Big(\frac{T}{2\pi}\Big)
     + 1.5\frac{\log T}{T^2}
     + \frac{2.97h}{8x}
\end{align*}
so that~\eqref{eq:A7} becomes
\begin{align}
|\Sigma_1|
&\leq 8x^{3/2}\Big(\frac{h}{8 x}\Big(1 - \frac{2}{\pi}\frac{x}{hT} + \frac{4}{\pi}\frac{x^2}{h^2T^2}\Big)\log\Big(\frac{T}{2\pi}\Big)
                   +1.5\frac{\log T}{T^2}
                   +\frac{2.97h}{8x}
              \Big)
    + \Big(\frac{\log^2 T}{\pi}+\frac{1}{20}\Big)\frac{h^2}{\sqrt{x}}            \notag\\
&=    h\sqrt{x}\Big(1 - \frac{2}{\pi}\frac{x}{hT} + \frac{4}{\pi}\frac{x^2}{h^2T^2}\Big)\log\Big(\frac{T}{2\pi}\Big)
    + 12x^{3/2}\frac{\log T}{T^2}
    + 2.97h\sqrt{x}
    + \Big(\frac{\log^2 T}{\pi}+\frac{1}{20}\Big)\frac{h^2}{\sqrt{x}}            \notag\\
&= h\sqrt{x}\log\Big(\frac{T}{2\pi}\Big)
    - \frac{2}{\pi}h\sqrt{x}\frac{x}{hT}\log\Big(\frac{T}{2\pi}\Big)
    + \frac{4}{\pi}h\sqrt{x}\frac{x^2}{h^2T^2}\log\Big(\frac{T}{2\pi}\Big)       \notag\\
&\hspace{0.5cm}
    + 2.97h\sqrt{x}
    + \frac{\log^2 T}{\pi}\frac{h^2}{\sqrt{x}}
    + 12x^{\frac{3}{2}}\frac{\log T}{T^2}
    + \frac{h^2}{20\sqrt{x}}.                                                    \label{eq:A8}
\end{align}

\section{Proof of Theorem~\ref{th:A1}}
\label{sec:A6}
Substituting~\eqref{eq:A8} into~\eqref{eq:A1} we get
\begin{multline*}
\sum_{|p-x|<h} \log p
\geq h
 - \Big[\sqrt{x}\log\Big(\frac{T}{2\pi}\Big)
         - \frac{2}{\pi}\sqrt{x}\frac{x}{hT}\log\Big(\frac{T}{2\pi}\Big)
         + \frac{4}{\pi}\sqrt{x}\frac{x^2}{h^2T^2}\log\Big(\frac{T}{2\pi}\Big)\\
         + 2.97\sqrt{x}
         + \frac{\log^2 T}{\pi}\frac{h}{\sqrt{x}}
         + 12x^{3/2}\frac{\log T}{hT^2}
         + \frac{h}{20\sqrt{x}}
   \Big]
 - 4(x+h)^{3/2} \frac{\log T}{\pi hT}
 - 0.002\sqrt{x} - 3\sqrt[3]{x} - \frac{2h}{\sqrt{x}}
 - \frac{3}{xh}.
\end{multline*}
Recalling that we have set $h= c\sqrt{x}\log x$, $T=\frac{\beta}{c}\frac{\sqrt{x}}{\log x}$ (so that
$hT=\beta x$), and estimating $(x+h)^{3/2}\leq x^{3/2}(1+2\frac{h}{x})$ (which holds whenever $h/x\leq
1.6$), we have that
\begin{multline*}
\sum_{|p-x|<h} \log p
\geq h
       - \sqrt{x}\log\Big(\frac{T}{2\pi}\Big)
       + \frac{2}{\pi}\frac{\sqrt{x}}{\beta}\log\Big(\frac{T}{2\pi}\Big)
       - \frac{4}{\pi}\Big(1+2\frac{h}{x}\Big)\sqrt{x}\frac{\log T}{\beta}
       - 3\sqrt{x}\\
       - \frac{4}{\pi}\frac{\sqrt{x}}{\beta^2}\log\Big(\frac{T}{2\pi}\Big)
       - \frac{\log^2 T}{\pi}\frac{h}{\sqrt{x}}
       - 12\frac{h}{\beta^2\sqrt{x}}\log T
       - 3\sqrt[3]{x}
       - 2.05\frac{h}{\sqrt{x}}
       - \frac{3}{xh},
\end{multline*}
or, upon gathering like terms, that
\begin{multline}\label{eq:A9}
\sum_{|p-x|<h} \log p
\geq h
       - \sqrt{x}\log T
       - \frac{2}{\pi}\frac{\sqrt{x}}{\beta}\log T
       - (3 - \log(2\pi))\sqrt{x}                                          \\
       - \frac{2}{\pi}\log(2\pi)\frac{\sqrt{x}}{\beta}
       - \frac{4}{\pi}\frac{\sqrt{x}}{\beta^2}\log\Big(\frac{T}{2\pi}\Big)
       - 3\sqrt[3]{x}
       - \frac{8}{\pi}\frac{\log T}{\beta}\frac{h}{\sqrt{x}}
       - \frac{\log^2 T}{\pi}\frac{h}{\sqrt{x}}
       - \Big(2.05+12\frac{\log T}{\beta^2}\Big)\frac{h}{\sqrt{x}}
       - \frac{3}{xh}.
\end{multline}

For this computation it is convenient to take $\beta=\beta(x)$ and diverging as $x$ goes to $\infty$. To
ensure the best result we have to set $\beta$ so that the sum $\log T + \frac{2}{\pi}\frac{\log T}{\beta}$
is minimised. This sum is, up to terms of lower order in $\beta$,
\[
\log \beta + \frac{\log x}{\pi\beta}.
\]
This last sum is minimum when
\[
\beta = \frac{1}{\pi} \log x.
\]
Thus we have $T=\frac{1}{\pi c}\sqrt{x}$. With this, the lower bound~\eqref{eq:A9} becomes
\begin{align*}
\sum_{|p-x|<h} \log p
&\geq
h
 - \frac{1}{2}\sqrt{x}\log x
 - (4 - \log(2c\pi^2))\sqrt{x}
 + o(\sqrt{x}).
\end{align*}
Using the fact that $c\geq 1/2$, which is the best we can do in this setting, in order to have a positive
lower bound it is sufficient to take
\[
h \geq \Big(\frac{1}{2} + \frac{d}{\log x}\Big)\sqrt{x}\log x
\]
for any $d > 4 - 2\log \pi = 1.7105\ldots$,
%
%G 4-2*log(Pi)
%G \\%4 = 1.710540228301199651713145297
%
when $x$ is large enough. Actually, the choice $d=1.72$ holds only for $x\geq \exp(590)\approx 2\cdot 10^{256}$.
On the contrary, the choice $d = 2$ holds for $x\geq 7.5\cdot 10^{8}$. Thus the claim asserting the
existence of a prime when $c=1/2 + 2/\log x$ is proved for $x\geq 7.5\cdot 10^{8}$. Moreover, the upper
bound $\log(x+h) \sum_{|p-x|< h}1 \geq \sum_{|p-x|< h}\log p$ and~\eqref{eq:A9} prove the existence of
$\sqrt{x}$ primes in $(x-(c+1)\sqrt{x}\log x, x+(c+1)\sqrt{x}\log x)$ for $x\geq 1.4\cdot 10^{5}$.
%
%\\ per verificare la seconda affermazione del teorema:
%
%G
%G minor5(d,u)=
%G {my(logT = 1/2*u - log(Pi*(1+1/2+d/u)), b=u/Pi, hrad=(1+1/2+d/u)*u);
%G  (hrad
%G    - logT
%G    - 2/Pi/b*logT
%G    - (3-log(2*Pi))
%G    - 2/Pi*log(2*Pi)/b
%G    - 4/Pi/b^2*(logT-log(2*Pi))
%G    - if (u < 1e5,my(x=exp(u));
%G          3/x^(1/6)
%G        + 8/Pi*logT/b*hrad/sqrt(x)
%G        + logT^2/Pi*hrad/sqrt(x)
%G        + (2.05+12/b^2*logT)*hrad/sqrt(x)
%G        + 3/hrad/x^2
%G      )
%G   )/(u + if(u<1e5,log(1+hrad/exp(u/2))))
%G };
%G
%G my(d=2);ploth(u=6,100,u*(minor5(d,u)-1))
%G \\%2 = [6.000000000000000000, 100.00000000000000000, -4.038590092531618048, 1.3372707451593377481]
%G my(d=2);solve(u=6,40,minor5(d,u)-1)     \\<---- verifico quando \`{e} maggiore di 1
%G \\%3 = 11.78405863087828745609036330
%G exp(11.8)
%G \\%34 = 133252.3529455309397353820661
%
Lastly, for $x\in[2,1.4\cdot 10^{5}]$ it is sufficient to check that $p_{k+1}-p_k \leq 2c\sqrt{p_k}\log p_k$
(which gives the claim for $x\in[p_k,p_{k+1}]$) when $k\leq 13010$.
%
%G \\ this proves the claim when $x\geq 2$:
%G \\ primepi(1.4e5) = 13010;  <--- 13010 suffices
%G {
%G   my(M = 1.4e5,Nk = 100,s = M/Nk);
%G   for(k = 0, Nk-1,
%G     my(n = primepi(s*k), p2, mul = 2*(0.5+1));
%G     forprime(p1 = s*k, s*(k+1),
%G       my(sp1 = sqrt(p1));
%G       n++;
%G       p2=prime(n+ceil(sp1));
%G       if(p2-p1>=mul*sp1*log(p1),print(n," ",p2," no"))
%G     );
%G     print("Slice ", k, " done")
%G   )
%G }
%
\section{An application}
\label{sec:A7}

On the Riemann hypothesis, Cram\'{e}r~\cite{Cramer} was the first to prove the bound $p_{n+1} - p_n \ll
\sqrt{p_n} \log p_n$, and he noted the implication that there exists some constant $\alpha>0$ such that
there will be a prime in the interval
\[
(n^2, (n+\alpha \log n)^2)
\]
for all sufficiently large $n$. This was intended for comparison to Legendre's conjecture that there is a
prime in the interval $(n^2, (n+1)^2)$ for all $n$. The following corollary of Theorem~\ref{th:A1} states
that one can take $\alpha = 1+o(1)$.

\begin{corollary} \label{cor1}
Assume RH. Then for every integer $n \geq 2$ there is a prime in the interval
\[
(n^2, (n+\alpha \log n)^2)
\]
where
\[
\alpha := 1 + \frac{2}{\log n} + \frac{1}{\log^2 n}.
\]
\end{corollary}

\begin{proof}
Let
\[
x := \frac{n^2 + (n+\alpha \log n)^2}{2}
\]
be the mid-point of the interval. We will prove that
$(x-c\sqrt{x}\log x,x+c\sqrt{x}\log x) \subseteq (n^2, (n+\alpha \log n)^2)$ with
$c = \frac{1}{2} + \frac{2}{\log x}$ so that the corollary will be a consequence of the theorem. Let
$\beta:=\alpha\frac{\log n}{n}$ and observe that $x=n^2(1+\beta+\frac{\beta^2}{2})$. We just need to prove
that $n^2\leq x-c\sqrt{x}\log x$. It holds if and only if
\[
\sqrt{x}(\log x + 4) \leq n^2(2\beta + \beta^2)
\]
which is equivalent to
\[
\sqrt{1+\beta+\frac{\beta^2}{2}}
\Big(2\log n + 4 + \log\Big(1+\beta+\frac{\beta^2}{2}\Big)\Big)
                               \leq 2\log n + 4 + \frac{2}{\log n} + n\beta^2.
\]
The last inequality is elementary and is true for any $n\geq 2$.
%G f(n)=my(be=(log(n)+2+1/log(n))/n);sqrt(1+be+be^2/2)*(2*log(n)+4+log(1+be+be^2/2))/(2*log(n)+4+2/log(n)+n*be^2);
%G ploth(n=2,10,f(n))
%G ploth(n=10,10^2,f(n))
%G ploth(n=10^2,10^4,f(n))
\end{proof}

Now, upon setting $n = p_k$ in the above corollary, it follows that there is a prime in the interval
\[
(p_k^2, (p_k + \alpha \log p_k)^2)
\]
for all $k \geq 1$. It should be noted, as $\alpha = 1+o(1)$ and the average gap between $p_k$ and
$p_{k+1}$ is $\log p_k$, that something can be said here about the existence of primes in the interval
$(p_k^2, p_{k+1}^2)$; this is related to the so-called Brocard conjecture predicting the existence of
four primes at least in this interval (see for instance Ribenboim~\cite[p.~248]{Ribenboim4}).

It was first proven by Cram\'{e}r~\cite{Cramer}, on RH, that the number of $n<x$ such that there is no prime
in the interval $(n^2,(n+1)^2)$ is $O(x^{2/3+\epsilon})$ (improved to $O(x^{1/2+\epsilon})$
unconditionally and to $O((\log x)^{2+\epsilon})$ on RH in~\cite{Bazzanella}),
%
% WARNING here! Bazzanella Th.1 proves that the number of integers $n$ with $[n^2,(n+1)^2]\in[0,N]$
% such that Legendre conj. is false for $n$ is at most $O(N^{1/4+\epsilon})$. Hence to compare this
% result with Cram\'{e}r's the normalization $x---> \sqrt{N}$ needs.
%
and from this it follows that there is a prime in \emph{almost all} intervals of the form $(p_k^2,
p_{k+1}^2)$. However, there may still be infinitely many exceptions, though the following corollary
assures us that the exceptions must occur when the prime gap is essentially less than the average gap.

\begin{corollary}
Assume RH. Suppose that $p_k$ and $p_{k+1}$ are consecutive primes satisfying
\[
p_{k+1} - p_k \geq \alpha \log p_k + \frac{\alpha^2 \log^2 p_k}{2 p_k}
\]
where $\alpha := \big(1 + \frac{1}{\log p_k}\big)^2$. Then there is a prime in the
interval $(p_k^2, p_{k+1}^2)$.
\end{corollary}

\begin{proof}
First, it follows that
\[
p_{k+1} - p_k > \frac{2 \alpha p_k \log p_k + \alpha^2 \log^2 p_k}{p_k + p_{k+1}}.
\]
It is straightforward to rearrange this so that
\[
p_{k+1}^2 > (p_k + \alpha \log p_k)^2
\]
and, with reference to Corollary~\ref{cor1}, this completes the proof.
\end{proof}

\appendix
\section{Second bound for $\Sigma_1$}
\label{app:A1}
Since $N(T)\leq \frac{T}{2\pi}\log T$ (see~\cite[Corollary~1]{TrudgianII}), from~\eqref{eq:A2} one has
\[
|\Sigma_1|
\leq \frac{2N(T)}{\sqrt{x-h}} \int_{x-h}^{x+h} K(x-u;h)\d u
=    \frac{2h^2}{\sqrt{x-h}} N(T)
\leq \frac{h^2T}{\pi\sqrt{x-h}}\log T,
\]
which is the way this sum is estimated in~\cite{Dudek}. We improve the result by proving the existence of
a cancellation for the sum $\sum_{|\gamma|\leq T} u^{i\gamma}$. The structure of the counting function
$N(T)$ alone, that is the fact that $N(T) = \frac{T}{2\pi}\log T + O(\log T)$, is not sufficient to
ensure a cancellation in $\sum_{|\gamma|\leq T}u^{i\gamma}$ \emph{for every} $u$. To see this, one can
consider a set of points generated in this way: in the neighborhood of every integer $n$ there is a cloud
of $\intpart{\frac{1}{\pi}\log n}$ points which are placed very close to $n$. Their counting function
satisfies the same formula as $N(T)$, size of the remainder included. For this set, however, one has
$\sum_{|\gamma|\leq T}u^{i\gamma} \gg T\log T$ when $u=e^{2\pi}$, and similarly for every $u=e^{2k\pi}$
when $k\in\N$ is small with respect to $T$.

Thus, we can furnish a cancellation essentially in two ways: either we assume some hypothesis about the
distribution of the imaginary parts of the zeroes of $\zeta(s)$ (for example the Pair Correlation
Conjecture, as done in~\cite{Heath-BrownGoldston} and in~\cite{LanguascoPerelliZaccagnini}, or the
stronger Gonek conjecture~\cite{Gonek}), or we try to prove a cancellation in some mean sense. The second
possibility appears promising since in our computation the estimated object appears naturally in an
integral and produces a result not depending on a further unproved hypothesis.\\
In this way we can prove Theorem~\ref{th:A1} with $c=0.6102$.

\subsection*{Cancellation in mean}
We let
\[
S_\alpha(T) := \sum_{|\gamma|\leq T}e^{i\alpha \gamma},
\]
keep the notations
\begin{align*}
T &= \frac{\beta}{c}\frac{\sqrt{x}}{\log x},
\hspace{-2.5cm}&
h &= c\sqrt{x}\log x,
\intertext{and introduce}
a &:= \log(x-h),
\hspace{-2.5cm}&
b &:= \log(x+h),\\
A &:= \frac{b-a}{2} = \frac{1}{2}\log\Big(\frac{x+h}{x-h}\Big),
\hspace{-2.5cm}&
B &:= \frac{a+b}{2} = \frac{1}{2}\log(x^2-h^2).
\end{align*}
Notice that $A \sim h/x \approx T^{-1}$ and $B\sim \log x$ as $x$ diverges to infinity.
\begin{proposition}~\label{prop:A1}
Assume RH. Suppose $\beta \geq 1$, $c\leq 1$ and $x\geq 2$. Then
\[
\int_{a}^b |S_\alpha(T)|^2 \d \alpha
\leq \frac{1}{\pi^2} F(AT) T\log^2\Big(\frac{T}{2\pi}\Big) + H(A,T)
\]
with
\[
F(y) := \frac{1}{y}\int_0^y\int_{-y}^y |\sinc(u-u')| \d u\d u'
\]
where $\sinc(x):= \frac{\sin x}{x}$, and
\[
H(A,T):= \frac{4}{\pi}(2 + AT)AT(R(T)+1)\log\Big(\frac{T}{2\pi}\Big)
         + 8A\Big(1 + AT + \frac{1}{3} (AT)^2\Big) (R(T) + 1)^2.
\]
\end{proposition}
\begin{remark*}
Once the orders of $h$ and $T$ as functions of $x$ are considered, the trivial bound and the new bound
are respectively
\[
\int_{a}^b |S_\alpha(T)|^2 \d \alpha
\leq \frac{2\beta}{\pi^2}\,T\log^2T
\quad
v.s.
\quad
\int_{a}^b |S_\alpha(T)|^2 \d \alpha
\leq \frac{1+o(1)}{\pi^2}F(\beta) T\log^2T,
\]
and it is easy to see that the second one improves on the first one for every $\beta>0$ as $T\to\infty$.
\end{remark*}
\begin{proof}
First, we have the series of equations:
\begin{align*}
\int_{a}^b |S_\alpha(T)|^2 \d \alpha
&= \Ree \int_{a}^b |S_\alpha(T)|^2 \d \alpha
 = \Ree\sum_{|\gamma|,|\gamma'|\leq T} \int_{a}^b e^{i\alpha(\gamma-\gamma')} \d \alpha                             \\
&= \Ree\sum_{|\gamma|,|\gamma'|\leq T} \frac{e^{ib(\gamma-\gamma')}-e^{ia(\gamma-\gamma')}}{i(\gamma-\gamma')}
 = 2\Ree\sum_{|\gamma|,|\gamma'|\leq T} e^{iB(\gamma-\gamma')}\frac{\sin\big(A(\gamma-\gamma')\big)}{\gamma-\gamma'}\\
& = 4\sum_{\substack{0<\gamma\leq T\\ -T\leq\gamma'\leq T}}
    \frac{\cos\big(B(\gamma-\gamma')\big)\sin\big(A(\gamma-\gamma')\big)}{\gamma-\gamma'}
\leq 4A\sum_{\substack{0<\gamma\leq T\\ -T\leq\gamma'\leq T}}
    |\sinc\big(A(\gamma-\gamma')\big)|.
\end{align*}
We will use below the following bounds for $\sinc(x)$:
\[
\|\sinc\|_\infty \leq 1,
\qquad
\|\sinc'\|_\infty \leq 1/2,
\qquad
\|\sinc''\|_\infty \leq 1/3.
\]
These are an immediate consequence of the representation $2\sinc(x) = \int_{-1}^1 e^{ixy}d y$.
%%%%%%%%%%%%%%%%
%The first few derivatives of $\sinc(x) := \frac{\sin x}{x}$ support the following conjecture:
%\[
%\|\sinc^{(n)}\|_\infty \leq \frac{1}{n+1}
%\]
%with equality when $n$ is even.
%%
%A quick proof comes from the representation $2\sinc(x) = \int_{-1}^1 e^{ixy}d y$, so that
%\[
%2\sinc^{(n)}(x) = \int_{-1}^1 (iy)^n e^{ixy}d y
%\qquad \implies \qquad
%2\|\sinc^{(n)}\|_\infty \leq 2\int_{0}^1 y^n d y = \frac{2}{n+1}.
%\]
%The same formula proves the equality when $n$ is even, because in that case $2\sinc^{(n)}(0) =
%2\int_{-1}^1 (iy)^n d y = \frac{2(-1)^{n/2}}{n+1}$.
%%%%%%%%%%%%%%%%
%
We thus write the double sum on zeros as a Stieltjes integral. Recalling that the imaginary part of the first
zero exceeds $2\pi$ we get
\begin{align*}
\int_{a}^b |S_\alpha(T)|^2 \d \alpha
\leq& 4A\int_{\substack{0<\gamma\leq T\\ - T\leq\gamma'\leq T}}
    |\sinc\big(A(\gamma-\gamma')\big)| \d N(\gamma)\d N(\gamma')\\
=& 4A\int_{\substack{2\pi<\gamma\leq T\\ 2\pi<\gamma'\leq T}}
    \Big(|\sinc\big(A(\gamma-\gamma')\big)| + |\sinc\big(A(\gamma+\gamma')\big)|\Big)\d N(\gamma)\d N(\gamma').
\end{align*}
To ease matters, we employ the notation
\[
f(t_1,t_2) := |\sinc(A(t_1 - t_2))| + |\sinc(A(t_1 + t_2))|
\]
which allows us to write that
\begin{align*}
\int_{a}^b |S_\alpha(T)|^2 \d \alpha
=& 4A\int_{\substack{2\pi<\gamma\leq T\\ 2\pi<\gamma'\leq T}}f(\gamma,\gamma')\d W(\gamma)\d W(\gamma')
   + 4A\int_{\substack{2\pi<\gamma\leq T\\ 2\pi<\gamma'\leq T}}f(\gamma,\gamma')\d W(\gamma)\d \Rp(\gamma')\\
 & + 4A\int_{\substack{2\pi<\gamma\leq T\\ 2\pi<\gamma'\leq T}}f(\gamma,\gamma')\d \Rp(\gamma)\d W(\gamma')
   + 4A\int_{\substack{2\pi<\gamma\leq T\\ 2\pi<\gamma'\leq T}}f(\gamma,\gamma')\d \Rp(\gamma)\d \Rp(\gamma').
\end{align*}
We write the sum of the above four integrals as $I +II + III + IV$ where the order is kept. It thus
remains to estimate separately the contribution of each integral. The first one produces the main term,
since
\begin{align*}
I
&\leq \frac{4A}{(2\pi)^2}
    \Big[\int_{\substack{2\pi<\gamma\leq T\\ 2\pi<\gamma'\leq T}}
         f(\gamma,\gamma') \d\gamma\d\gamma'
    \Big] \log^2\Big(\frac{T}{2\pi}\Big)\\
&\leq \frac{1}{\pi^2 A}
    \Big[\int_{\substack{0<u\leq AT\\ -AT\leq u'\leq AT}}
         |\sinc(u-u')| \d u\d u'
    \Big] \log^2\Big(\frac{T}{2\pi}\Big)
= \frac{1}{\pi^2} F(AT) T\log^2\Big(\frac{T}{2\pi}\Big).
\end{align*}
In estimating the integral $II$, an application of integration by parts gives (note that
$\partial_{\gamma'}|\sinc(A(\gamma\pm\gamma'))|$ has only jump singularities, so the formula still holds)
\begin{align*}
II
&= \frac{4A}{2\pi}\int_{2\pi}^T \Big[\int_{2\pi}^T f(\gamma,\gamma')\log\Big(\frac{\gamma}{2\pi}\Big)\d \gamma\Big]\d \Rp(\gamma')\\
&= \frac{2A}{\pi}\Big[\int_{2\pi}^T f(\gamma,\gamma')\log\Big(\frac{\gamma}{2\pi}\Big)\d \gamma\Big]\Rp(\gamma')\Big|_{2\pi}^T
  -\frac{2A}{\pi}\int_{2\pi}^T \Big[\int_{2\pi}^T \partial_{\gamma'}f(\gamma,\gamma')\log\Big(\frac{\gamma}{2\pi}\Big)\d \gamma\Big]\Rp(\gamma')\d\gamma'.
\intertext{We can estimate it as (recall that $\Rp(2\pi)=1$)}
&\leq \frac{2A}{\pi}\int_{2\pi}^T 2\|\sinc\|_\infty\log\Big(\frac{\gamma}{2\pi}\Big)\d \gamma (R(T) + 1)
  +\frac{2A^2}{\pi}\int_{2\pi}^T \int_{2\pi}^T 2\|\sinc'\|_\infty\log\Big(\frac{\gamma}{2\pi}\Big)|\Rp(\gamma')|\d \gamma\d\gamma'\\
&\leq \frac{4}{\pi}AT(R(T)+1)\log\Big(\frac{T}{2\pi}\Big)
  +\frac{2}{\pi}(AT)^2R(T)\log\Big(\frac{T}{2\pi}\Big)\\
&\leq \frac{2}{\pi}(2 + AT)AT(R(T)+1)\log\Big(\frac{T}{2\pi}\Big).
\end{align*}
The contribution of $III$ equals that of $II$, for we note the symmetry of the integral under the
transposition $\gamma\leftrightarrow \gamma'$. And so, lastly we have
\begin{align*}
IV
=& 4A\int_{2\pi}^T \Big[\int_{2\pi}^T f(\gamma,\gamma')\d \Rp(\gamma)\Big]\d \Rp(\gamma')\\
=& 4A \int_{2\pi}^T\Big[f(\gamma,\gamma')\Rp(\gamma)\Big|_{2\pi}^T\Big]\d \Rp(\gamma')
  -4A\int_{2\pi}^T \Big[\int_{2\pi}^T \partial_{\gamma}\Big(f(\gamma,\gamma')\Big)\Rp(\gamma)\d \gamma\Big]\d \Rp(\gamma')\\
=& 4A \int_{2\pi}^T f(T,\gamma')\Rp(T)\d \Rp(\gamma')
  -4A \int_{2\pi}^T f(2 \pi,\gamma') \d \Rp(\gamma')                                     \\
 &-4A\int_{2\pi}^T \Big[\int_{2\pi}^T \partial_{\gamma}\Big(f(\gamma,\gamma')\Big)\Rp(\gamma)\d \gamma\Big]\d \Rp(\gamma'),
\intertext{where a second integration by parts gives}
IV
=& 4Af(T,\gamma') \Rp(T)\Rp(\gamma')\Big|_{2\pi}^T -4A \int_{2\pi}^T\partial_{\gamma'}(f(T,\gamma')) \Rp(T)\Rp(\gamma')\d\gamma'\\
 &-4A f(2\pi,\gamma')\Rp(\gamma')\Big|_{2\pi}^T +4A \int_{2\pi}^T\partial_{\gamma'}(f(2 \pi, \gamma')) \Rp(\gamma')\d\gamma'\\
 &-4A\Big[\int_{2\pi}^T \partial_{\gamma}\Big(f(\gamma,\gamma')\Big)\Rp(\gamma)\d \gamma\Big]\Rp(\gamma')\Big|_{2\pi}^T
  +4A\int_{2\pi}^T \Big[\int_{2\pi}^T \partial_{\gamma'}\partial_{\gamma}\Big(f(\gamma,\gamma')\Big)\Rp(\gamma)\d \gamma\Big]\Rp(\gamma')\d \gamma'.
\intertext{Thus, one may establish the bound}
IV
\leq& \big[8A R^2(T) + 8A R(T)\big]
  +4A^2 \int_{2\pi}^T 2\|\sinc'\|_\infty R(T)|\Rp(\gamma')|\d\gamma'\\
 &+\big[8A R(T) + 8A\big]
  +4A^2 \int_{2\pi}^T 2\|\sinc'\|_\infty |\Rp(\gamma')|\d\gamma'\\
 &+4A^2\Big[\int_{2\pi}^T 2\|\sinc'\|_\infty |\Rp(\gamma)|\d \gamma\Big](R(T)+1)
  +4A^3\int_{2\pi}^T \int_{2\pi}^T 2\|\sinc''\|_\infty |\Rp(\gamma)\Rp(\gamma')|\d \gamma\d \gamma'.
  \intertext{Estimating the integrals, one has that}
IV
\leq& 8A R^2(T) + 8A R(T)
  +4A^2 T R^2(T)
  +8A R(T) + 8A
  +4A^2 T R(T)\\
 &+4A^2 T R(T)(R(T)+1)
  +\frac{8}{3}A^3 T^2 R^2(T)\\
%
%=& 8A (R(T)^2 + 2R(T) +1)
%+4A^2T(2R^2(T) + 2R(T))
%+\frac{8}{3} A^3T^2 R^2(T)\\
%
%=& 8A (R(T) + 1)^2
%+8A^2T(R(T)| + 1)R(T)
%+\frac{8}{3} A^3T^2 R^2(T)\\
%
=& 8A \Big((R(T) + 1)^2 + AT(R(T) + 1)R(T) + \frac{1}{3} (AT)^2 R^2(T)\Big)\\
\leq& 8A\Big(1 + AT + \frac{1}{3} (AT)^2\Big) (R(T) + 1)^2.
\end{align*}
Therefore, the contribution of $II$, $III$ and $IV$ is bounded by
\[
2\cdot\frac{2}{\pi}(2 + AT)AT(R(T)+1)\log\Big(\frac{T}{2\pi}\Big)
+8A\Big(1 + AT + \frac{1}{3} (AT)^2\Big) (R(T) + 1)^2,
\]
which is $H(A,T)$.
\end{proof}

\subsection*{Estimation of $\Sigma_1$}
We now use the above result on cancellation to estimate the first sum over the zeroes. The Cauchy-Schwarz
inequality yields the bound
\begin{align}
|\Sigma_1|
&%\leq \int_{x-h}^{x+h} K(x-u;h)\Big|\sum_{|\gamma|\leq T}u^{i\gamma}\Big|\frac{\d u}{\sqrt{u}}
 \leq \int_{x-h}^{x+h} K(x-u;h)\Big|\sum_{|\gamma|\leq T}e^{i\gamma\log u}\Big|\frac{\d u}{\sqrt{u}}                    \notag\\
&= \int_{\log(x-h)}^{\log(x+h)} e^{\alpha/2}K(x-e^{\alpha};h)\Big|\sum_{|\gamma|\leq T}e^{i\alpha\gamma}\Big|\d \alpha  \notag\\
&\leq \Big[\int_{\log(x-h)}^{\log(x+h)} e^{\alpha}K^2(x-e^{\alpha};h)\d \alpha\Big]^{1/2}
     \Big[\int_{\log(x-h)}^{\log(x+h)} |S_\alpha(T)|^2\d \alpha\Big]^{1/2},                                             \notag\\
\intertext{and so we have that}
| \Sigma_1|
&\leq \Big[\int_{-h}^{h} K^2(u;h)\d u\Big]^{1/2}
     \Big[\int_{\log(x-h)}^{\log(x+h)} |S_\alpha(T)|^2\d \alpha\Big]^{1/2}                      \notag\\
&= \sqrt{\frac{2}{3}}\,h^{3/2}
     \Big[\int_{\log(x-h)}^{\log(x+h)} |S_\alpha(T)|^2\d \alpha\Big]^{1/2}.                     \notag
\intertext{We can now apply Proposition~\ref{prop:A1} to get the estimate}
|\Sigma_1|
&\leq \sqrt{\frac{2}{3}}\,h^{3/2}
     \Big(\frac{1}{\pi^2}F(AT)T\log^2\Big(\frac{T}{2\pi}\Big) + H(A,T)\Big)^{1/2}               \notag\\
&= \frac{1}{\pi}
    \Big(\frac{2}{3}F(AT) + \frac{2\pi^2}{3}\frac{H(A,T)}{T\log^2(T/2\pi)}\Big)^{1/2}
    h\sqrt{hT}\log\Big(\frac{T}{2\pi}\Big)                                                      \notag\\
&= \frac{1}{\pi}
    \Big(\frac{2\beta}{3}F(AT) + \frac{2\beta\pi^2}{3}\frac{H(A,T)}{T\log^2(T/2\pi)}\Big)^{1/2}
    h\sqrt{x}\log\Big(\frac{T}{2\pi}\Big).                                                      \label{eq:A4}
\end{align}

We now proceed to prove the analog of Theorem~\ref{th:A1}.

\subsection*{First claim}
We want to prove that there is a prime in $(x-c\sqrt{x}\log x, x+c\sqrt{x}\log x)$ with $c=0.6102$.
From~\eqref{eq:A1} and the bound~\eqref{eq:A4} for $\Sigma_1$ we get
\begin{multline}\label{eq:A5}
h^{-1}\sum_{|p-x|<h} \log p
\geq 1
 - \Big(\frac{2\beta}{3}F(AT) + \frac{2\beta\pi^2}{3}\frac{H(A,T)}{T\log^2(T/2\pi)}\Big)^{1/2}\frac{\sqrt{x}}{\pi h}\log\Big(\frac{T}{2\pi}\Big)
 - 4(x+h)^{3/2} \frac{\log T}{\pi h^2T}\\
 - 0.002\frac{\sqrt{x}}{h} - 3\frac{\sqrt[3]{x}}{h} - \frac{2}{\sqrt{x}}
 - \frac{3}{xh^2}
\end{multline}
when $x\geq 121$.
When $x$ diverges to infinity, Inequality~\eqref{eq:A5} becomes
\begin{equation}\label{eq:A6}
h^{-1}\sum_{|p-x|<h} \log p
\geq
 1 - \frac{\alpha}{c} + \Big(\frac{2\alpha}{c} + o(1)\Big)\frac{\log\log x}{\log x},
\end{equation}
where $\alpha := \frac{1}{\pi}\big(\sqrt{\frac{\beta}{6}F(\beta)} + \frac{2}{\beta}\big)$, uniformly for
$\beta$ and $c$ in any compact set of $(0,+\infty)$.
The minimum of $\alpha$ is attained for $\beta = \beta_{\min} := 2.4934\ldots$, and is $\alpha_{\min} :=
0.61019\ldots$. Thus, setting $\beta=\beta_{\min}$ and $c=\alpha_{\min}$, the right hand side
of~\eqref{eq:A6} is positive when $x$ is large enough.
%
%G ABSSinc(x)=if(x!=0,abs(sin(x)/x),0);
%G F(b)=intnum(v=0,b,intnum(u=-b,b,ABSSinc(u-v)))/b;
%G FF(b)=(sqrt(b*F(b)/6)+2/b)/Pi;
%G ploth(b=1,4,FF(b),,100)
%G \\%7 = [1.0000000000000000000, 4.000000000000000000, 0.6101995806559580293, 0.8107073999188270097]
%G \\ the minimum:
%G ploth(b=1,4,FF(b),,100)
%G \\%7 = [1.0000000000000000000, 4.000000000000000000, 0.6101995806559580293, 0.8107073999188270097]
%G ploth(b=2,3,FF(b),,100)
%G \\%8 = [2.000000000000000000, 3.000000000000000000, 0.6101968534808974720, 0.6212991540509831268]
%G ploth(b=2.45,2.55,FF(b),,100)
%G \\%9 = [2.450000000000000178, 2.549999999999999822, 0.6101967191291555182, 0.6103060010120414658]
%G ploth(b=2.48,2.5,FF(b),,100)
%G \\%10 = [2.479999999999999982, 2.500000000000000000, 0.6101967154886323020, 0.6102035057563122145]
%G ploth(b=2.49,2.5,FF(b),,100)
%G \\%11 = [2.490000000000000213, 2.500000000000000000, 0.6101967154886323020, 0.6101983235581985854]
%G ploth(b=2.492,2.494,FF(b),,100)
%G %28 = [2.491999999999999993, 2.494000000000000217, 0.6101965491467539193, 0.6101966719405935713]
%G ploth(b=2.4934,2.4936,FF(b),,100)
%G %29 = [2.493399999999999839, 2.493599999999999817, 0.6101965491467539193, 0.6101965520844976343]
%
Inserting $\beta=2.493$ and $c=0.6102$ directly into~\eqref{eq:A5} one obtains an inequality where the
function appearing on the right hand side is positive whenever $x\geq 16000$, so that the claim is
proved in this range.
%
%G \\ Test the theorem
%G
%G ABSSinc(x)=if(x!=0,abs(sin(x)/x),1);
%G F(b)=intnum(v=0,b,intnum(u=-b,b,ABSSinc(u-v)))/b;
%
%G R(T) = 0.112*log(T) + 0.278*log(log(T)) + 2.51 + 0.2/T + 7/8;
%
%G H(A,T)=
%G {
%G 4/Pi*(2+A*T)*A*T*(R(T)+1)*log(T/(2*Pi))
%G + 8*A*(1+A*T+(A*T)^2/3)*(R(T)+1)^2
%G };
%
%G DX(b,c,x)=
%G {
%G my(h = c*sqrt(x)*log(x), T = b/c*sqrt(x)/log(x), A = 1/2*log((x+h)/(x-h)));
%G 1
%G - sqrt(2/3*b*F(A*T) + (2*b*Pi^2)/3*H(A,T)/(T*log(T/(2*Pi))^2)) * sqrt(x)/(Pi*h)*log(T/(2*Pi))
%G - 4*(x+h)^(3/2)*log(T)/(Pi*h^2*T)
%G - 0.002*sqrt(x)/h
%G - 3*x^(1/3)/h
%G - 2/sqrt(x)
%G - 3/(x*h^2)
%G };
%G
%G my(b=2.4934,c=0.6102);ploth(x=100,10^6,DX(b,c,x),,100)
%G \\%33 = [100.00000000000000000, 1000000.0000000000000, -3.112845391186497945, 0.2526417290560152851]
%G my(b=2.4934,c=0.6102);ploth(x=10^4,2*10^4,DX(b,c,x),,100)
%G \\%35 = [10000.000000000000000, 20000.000000000000000, -0.06859430617297923583, 0.03532773127678270897]
%G
%G my(b=2.4934,c=0.6102);solve(x=10^4,2*10^4,DX(b,c,x))
%G %42 = 15389.85972143089990067581221
%G
%
Lastly, for $x\in[2,16000]$ it is sufficient to check that $p_{k+1}-p_k \leq 2c\sqrt{p_k}\log p_k$
(which gives the claim for $x\in[p_k,p_{k+1}]$) when $k\leq 2000$.
%
%G \\ this proves the claim when $x\geq 2$:
%G \\ primepi(16000) = 1862;  <--- 2000 suffices
%G my(n=0);forprime(p1=2,16000,n++;p2=nextprime(p1+1);if(p2-p1<2*0.6*sqrt(p1)*log(p1),,print(n," ",p2," no")))
%

\subsection*{Second claim}
We want to prove that there are at least $\sqrt{x}$ primes in $(x-(c+1)\sqrt{x}\log x,x+(c+1)\sqrt{x}\log
x)$ with $c=0.6102$. Since
\[
h^{-1}\sum_{|p-x|<h} \log p \leq \frac{1+O(\frac{\log x}{\sqrt{x}})}{c\sqrt{x}}\sum_{|p-x|<h} 1,
\]
from~\eqref{eq:A6} we also get that
\[
\sum_{|p-x|<h} 1
\geq
 \Big(c - \alpha + (2\alpha + o(1))\frac{\log\log x}{\log x}\Big)\sqrt{x}.
\]
In particular, setting $\alpha=\alpha_{\min}$ and $c=\alpha_{\min}+1$ this shows that
\[
\sum_{|p-x|<(\alpha_{\min}+1)\sqrt{x}\log x} 1
\geq \sqrt{x},
\]
when $x$ is large enough. Once again, choosing these values directly in~\eqref{eq:A6} one gets an
explicit inequality which can be proved for $x\geq 1500$, proving the statement in this range.
%
%G my(b=2.4934,c=0.6102+1);ploth(x=10^2,10^4,h=c*sqrt(x)*log(x);DX(b,c,x)*c*log(x)/log(x+h),,100)
%G \\%44 =  [100.00000000000000000, 10000.000000000000000, -0.4448218077155521866, 5.304386097262178801]
%G my(b=2.4934,c=0.6102+1);solve(x=10^3,10^4,h=c*sqrt(x)*log(x);DX(b,c,x)*c*log(x)/log(x+h))
%G \\%45 = 1463.102525262685869859292014
%
The claim for $x\in[2,1600]$ may be checked directly by noticing that $p_{n+\intpartup{\sqrt{p_n}}} - p_{n
\vphantom{\intpartup{\sqrt{p_n}}}} \leq 2(c+1)\sqrt{p_n}\log p_n$ (giving the claim for $x\in[p_{n
\vphantom{\intpartup{\sqrt{p_n}}}}, p_{n+\intpartup{\sqrt{p_n}}}]$) for $n=1,\ldots,251$.
%
%G \\ this proves the claim when $x\geq 2$:
%G \\ primepi(1600) = 251;  <--- 251 suffices
%G my(n=0);forprime(p1=2,1600,n++;p2=prime(n+ceil(sqrt(p1)));if(p2-p1>= 2*(0.6+1)*sqrt(p1)*log(p1),print(n," ",p2," no")))
%

\begin{acknowledgements}
All computations have been done using PARI/GP~\cite{PARI2}.\\
We are very grateful to the referee for her/his careful reading and suggestions which improved the
presentation.
\end{acknowledgements}

%\bibliographystyle{amsplain}
%\bibliography{f:/books}

\end{document}